\documentclass[a4paper,12pt]{article}
\usepackage{amsfonts}
\usepackage{pifont}
\usepackage{caption}
\usepackage{graphicx, subfig}
\usepackage{bm}
\usepackage{latexsym,amsmath,amssymb,cite,amsthm}
\usepackage{color,eucal,enumerate,mathrsfs}
\usepackage[normalem]{ulem}
\usepackage{amsmath}
\usepackage[pagewise]{lineno}
\usepackage[pagebackref=true, colorlinks, linkcolor=blue,  anchorcolor=blue,citecolor=blue]{hyperref}

\renewcommand{\d}{\mathrm{d}}
\newcommand{\mcpkn}{{\rm MCP}(K, N)}
\newcommand{\mm}{\mathfrak m}
\newtheorem{proposition}{Proposition}[section]
\newtheorem{theorem}[proposition]{Theorem}

\newtheorem{definition}[proposition]{Definition}

\theoremstyle{definition}

\theoremstyle{remark}
\newtheorem{remark}[proposition]{Remark}

\numberwithin{equation}{section}
\allowdisplaybreaks

\newcommand{\N}{\mathbb{N}}

\newcommand{\R}{\mathbb{R}}

\title{{\bf Sharp uncertainty principles on metric measure spaces}
}

\begin{document}


\author{Bang-Xian Han\thanks{School of   Mathematical Sciences, University of Science and Technology of China, 230026, Hefei, China, hanbangxian@ustc.edu.cn   }
\and Zhefeng Xu
\thanks{School of Mathematical Sciences, University of Science and Technology of China,  230026, Hefei, China, xzf1998@mail.ustc.edu.cn}}

\date{\today}

\maketitle
\begin{abstract}
  We study the Heisenberg--Pauli--Weyl uncertainty principle and the Caffarelli--Kohn--Nirenberg interpolation inequalities,   on  metric measure spaces satisfying measure contraction property.  Using  localization techniques,  we show that  these inequalities are valid only on volume cones. 
\end{abstract}

\maketitle

\tableofcontents


\section{Introduction}
    The classical Heisenberg uncertainty principle was first established by Heisenberg \cite{Heisenberg}, who brings a fundamental problem in quantum mechanics to the points: \emph{the position and the momentum of particles cannot be both determined explicitly but only in a probabilistic sense with a certain uncertainty}.  A few years later, Pauli and Weyl \cite{Weyl} described it by rigorous mathematical formulation, which states that a function itself and its Fourier transform cannot be well localized simultaneously. The Heisenberg--Pauli--Weyl uncertainty principle on the Euclidean space is described by the following inequality
\begin{equation}\label{equation1-1}
	\left(\int_{\R^n} |\nabla u(x)\,|^2\,\mathrm{d} x \right) \left(\int_{\R^n} |x|^2 u^2(x)\,\mathrm{d} x \right) \geq \frac{n^2}{4}\left(\int_{\R^n} u^2(x)\,\mathrm{d} x\right)^2,\quad \forall u\in C^\infty_0(\R^n)
\end{equation}    
    where $\frac{n^2}{4}$ is sharp and the  extremals are $u_\lambda(x)=e^{-\lambda|x|^2}, \lambda>0$.

    The Heisenberg--Pauli--Weyl uncertainty principle only had sporadic developments in the fifty years after the initial work in the 1930's, followed by a steady stream of results in the last forty years. It has been studied in various contexts, in the 1980's and 1990's, Fefferman \cite{MR0707957} and Nahmod \cite{MR1255277} solved eigenvalue estimates for differential and pseudodifferential operators by SAK principle.  We  refer to a survey written by Folland and Sitaram \cite{MR1448337},  where they gave an overview of the history and the relevance of (\ref{equation1-1}) in the last century.   At the beginning of this century, Ciatti, Ricci and Sundari \cite{MR2355602} extended this principle to positive self-adjoint operators on measure spaces, and in the following years, Erb \cite{Erb,MR2652457}  and Kombe-\"Ozaydin  \cite{MR2538592,MR3074365} proved a sharp uncertainty principle on Riemannian manifolds by operator theoretic approach, Huang,  Krist\'aly and Zhao \cite{HKZ-uncertain} got a sharp uncertainty principle  on Finsler manifolds. In the context of metric measure spaces, Okoudjou, Saloff-Coste and Teplyaev  \cite{MR2386249} proved a weak uncertainty principle,     Mart\'{i}n--Milman \cite{MR3475459} obtained an $L^1$-uncertainty principle with isoperimetric weights.

\medskip

Inspired by a recent work of Krist\'{a}ly \cite{MR3862150},    where he revealed the rigidity of the Heisenberg--Pauli--Weyl uncertainty principle on Riemannian manifolds with non-negative Ricci curvature, we realize that  similar rigidity holds  on a large family of metric measure spaces, called essentially non-branching  ${\rm MCP}(0,N)$ spaces. Important examples include Riemannian manifolds with non-negative Ricci curvature and their Gromov--Hausdorff  limits, RCD$(0, N)$ spaces and many ideal sub-Riemannian manifolds including generalized H-type groups, the Grushin plane and Sasakian structures.

    Let $(X,d)$ be a complete and separable metric space and $\mm$ be a Radon measure  on $X$ with full support. Fix $x_0\in X$ and consider the Heisenberg--Pauli--Weyl uncertainty principle (HPW for short) on $(X,d,\mm)$ in  the following form: for any $u\in \text{Lip}_c(X, d)$,
\begin{equation}\label{first}\tag*{$(\mathrm{HPW})_{x_0}$}
	\left(\int_X |\,\text{lip}\,u\,|^2\,\mathrm{d} \mm\right) \left(\int_X d^2_{x_0} u^2\,\mathrm{d} \mm\right) \geq \frac{N^2}{4}\left(\int_X u^2\,\mathrm{d} \mm\right)^2
\end{equation}
	where $d_{x_0}(x):=d(x_0,x)$ is the distance function from $x_0$ and Lip$_c(X, d)$ is the space of Lipschitz functions with compact support,  and
\begin{equation*}
	\text{lip}\,u(x):=\limsup_{y\rightarrow x} \frac{|u(y)-u(x)|}{d(x,y)}
\end{equation*}
    is the local Lipschitz constant of $u$ at $ x\in X$.
    
    In our first theorem ,we generalize Krist\'{a}ly's result  \cite{MR3862150} to metric measure spaces.
    
\begin{theorem}\label{theorem1}
	Let $(X,d,\mm)$ be an essentially non-branching metric measure space satisfying ${\rm MCP}(0,N)$ for some $N\in(1,\infty)$. Then the following statements are equivalent:
	\begin{itemize}
	\item [(a)]  {\ref{first}} holds for some $x_0\in X$ and $\frac{N^2}{4}$ is sharp.
	\item [(b)] $ (X,d,\mm)$ is a a volume cone.
	\end{itemize}
\end{theorem} 

    Next we prove the rigidity of the  Caffarelli--Kohn--Nirenberg interpolation inequality (CKN for short) in the setting of non-smooth metric measure spaces. The classical CKN in the Euclidean setting was first proposed in \cite{MR0768824}, then Lin \cite{MR0864416} generalized it to include derivatives of any order. It is known that the CKN contains the classical Sobolev inequality and Hardy inequality as special cases, which have played important roles in many applications. 
    
   	Let $N,p,q\in \R$ be such that 
\begin{equation}\label{equation1-2}
   		0<q<2<p \,\,\text{and}\,\, 2<N<\frac{2(p-q)}{p-2}.
\end{equation}
    Fix $x_0\in X$, for all $u\in \text{Lip}_c(X,d)$,
\begin{equation*}\label{second}\tag*{$(\mathrm{CKN})_{x_0}$}
	\left(\int_X |\,\text{lip}\,u\,|^2\,\mathrm{d} \mm\right)\left(\int_X \frac{|u|^{2p-2}}{d^{2q-2}_{x_0}}\,\mathrm{d} \mm\right) \geq \frac{(N-q)^2}{p^2}\left(\int_X \frac{|u|^p}{d^q_{x_0}}\,\mathrm{d} \mm\right)^2.
\end{equation*}
    An endpoint of \ref{second} is exactly the \ref{first} whenever $p\rightarrow 2$ and $q\rightarrow 0$. In the  Euclidean setting, Xia \cite{MR2350131} proved the sharpness of $\frac{(N-q)^2}{p^2}$ and the existence of a class of extremals
\begin{equation}\label{equation1-3}
	u_\lambda(x)=(\lambda+|x-x_0|^{2-q})^{\frac{1}{2-p}},\, \lambda>0.
\end{equation}

Similar to Theorem \ref{theorem1},   $(\mathrm{CKN})_{x_0}$ is also rigid.

\begin{theorem}\label{theorem2}
	Let $N,p,q$ be real numbers satisfying \eqref{equation1-2}  and $(X,d,\mm)$ be an essentially non-branching metric measure space satisfying ${\rm MCP}(0,N)$.  Then the following statements are equivalent:
\begin{itemize}
	\item [(a)]   \ref{second} holds for some $x_0\in X$ and $\frac{(N-q)^2}{p^2}$ is sharp.
	\item [(b)] $(X,d,\mm)$ is a volume cone.
	\end{itemize} 
\end{theorem}
 
  \noindent   {\it Plan of the paper}. In Section \ref{section2}, we introduce some basic concepts and results on metric measure spaces. In Section \ref{section3}, we prove the rigidity of  \ref{first} on metric measure space in Theorem \ref{theorem1},  with the help of  the needle decomposition  and the Generalized Bishop--Gromov inequality. In Section \ref{section4},  we prove the rigidity of  \ref{second} in Theorem \ref{theorem2}.

    \medskip
    
    \noindent \textbf{Declaration.}
{The  authors declare that there is no conflict of interest and the manuscript has no associated data.}

\medskip

\noindent \textbf{Acknowledgement.} This project is supported in part by  the Ministry of Science and Technology of China, through the Young Scientist Programs (No. 2021YFA1000900  and 2021YFA1002200), and NSFC grant (No.12201596). The authors want to thank Alexandru Krist\'{a}ly for his suggestions on the first draft of this paper.
  
\section{Preliminaries}\label{section2}
  In this paper,  $(X,d)$ is a Polish space (i.e. a complete and separable metric space), and $\mm$ is a Radon measure on $X$ such that $0<\mm(U)<\infty$ for any non-empty bounded open set $U\in X$ (i.e. supp $\mm=X$). The triple $(X,d,\mm)$ is said to be a metric measure space.

    Denote by 
\begin{equation*}
	\text{Geo}(X):=\Big\{\gamma\in C([0,1],X):d(\gamma_s,\gamma_t)=|s-t|d(\gamma_0,\gamma_1), \,\,\forall s,t\in [0,1]\Big\}
\end{equation*}
   the space of constant-speed geodesics. We  assume that $(X,d)$ is a geodesic space, this means,  for each $x,y \in X$ there exists $\gamma\in \text{Geo}(X)$ so that $\gamma_0=x,\gamma_1=y$.

  Denote by $\mathscr{P}(X)$ the space of all Borel probability measures on X and by $\mathscr{P}_2(X)$ the space of probability measures with finite second moment. Endow $\mathscr{P}_2(X)$ with the $L^2$-Kantorovich--Wasserstein distance $W_2$ defined as follows: for any $ \mu_0,\mu_1 \in \mathscr{P}_2(X)$,   set
\begin{equation}\label{equation2-1}
	W^2_2(\mu_0,\mu_1):= \inf\limits_{\pi} \int_{X\times X}  d^2(x,y)\, \pi(\mathrm{d}x\mathrm{d}y),
\end{equation}
    where the infimum is taken over all $\pi \in\mathscr{P}(X\times X)$ with $\mu_0$ and $\mu_1$ as the first and the second marginals.

For any geodesic $(\mu_t)_{t\in[0,1]}$ in $(\mathscr{P}_2(X),W_2)$,  there is $\nu \in \mathscr{P}(\text{Geo}(X))$, so that $$(e_t)_\sharp \nu=\mu_t~ \text{ for all}~t\in [0,1]$$where $e_t$ denote the evaluation map 
\begin{equation*}
	e_t:\text{Geo}(X) \rightarrow X,\quad e_t(\gamma):=\gamma_t.
\end{equation*}  We denote by OptGeo($\mu_0,\mu_1$) the space of all $\nu \in \mathscr{P}(\text{Geo}(X))$ for which $(e_0,e_1)_\sharp \nu$ realizes the minimum in (\ref{equation2-1}), such a $\nu$ will be called dynamical optimal plan. If $(X,d)$ is geodesic, then  OptGeo($\mu_0,\mu_1$) is non-empty for any $\mu_0,\mu_1 \in \mathscr{P}_2(X)$.
    
    \medskip
    
    Next we briefly recall the synthetic notions of lower Ricci curvature bounds, for more details we refer to  \cite{MR2610378,MR2480619,MR2237206,MR2237207}. Given $K\in \R $ and $N\geq0$.   For $ (t,\theta)\in [0,1]\times\R_+$, we define the distortion coefficients as
\begin{equation*}
\sigma^{(t)}_{K,N}(\theta):=
\begin{cases}
    \infty &\text{if}\,\,K\theta^2\geq N \pi^2,\\
    \frac{\sin(t\theta\sqrt{K/N})}{\sin(\theta\sqrt{K/N})} &\text{if}\,\,0<K\theta^2< N \pi^2,\\
    t &\text{if}\,\,K\theta^2<0\,\,\text{and}\,\,N=0,\text{or if}\,\, K\theta^2=0,\\
    \frac{\sinh(t\theta\sqrt{-K/N})}{\sinh(\theta\sqrt{-K/N})} &\text{if}\,\,K\theta^2\leq0\,\,\text{and}\,\, N>0.\\
\end{cases}
\end{equation*}	

    We also set, for $K\in \R, N\in [1,\infty)$ and $(t,\theta)\in [0,1]\times\R_+$,
\begin{equation*}
	\tau^{(t)}_{K,N}(\theta):=t^{\frac{1}{N}}\sigma^{(t)}_{K,N-1}(\theta)^{\frac{N-1}{N}}.
\end{equation*}

Recall the following notion of essentially non-branching  \cite{MR3216835}.
\begin{definition}
	A set  $G \in \mathrm{Geo}(X) $ is a set of non-branching geodesics if  for any $\gamma^1,\gamma^2 \in G$, it holds:
\[
\exists t \in (0,1) ~~\text{s.t.}~ ~\forall  s\in [0, t]~ \gamma_s^1 =\gamma_s^2 \Rightarrow  \forall s \in [0,1] \;\; \gamma_s^1 =\gamma_s^2 . 
\]
\end{definition}  
  
\begin{definition}
     A metric measure space $(X,d,\mm)$ is called essentially non-branching if  for any $\mu_0,\mu_1 \in \mathscr{P}_2(X)$ with $\mu_0, \mu_1\ll \mm$, any element of $\mathrm{OptGeo}(\mu_0,\mu_1)$ is concentrated on a set of non-branching geodesics.
\end{definition}
    If $(X,d)$ is a smooth Riemannian manifold, then any subset $ G \subset \mathrm{Geo}(X) $ is a set of non-branching geodesics.  More generally,  it is known that RCD spaces are essentially non-branching.
    
   \medskip
    
    Next we recall the definition of MCP$(K,N)$ given independently in \cite{MR2341840} and \cite{MR2237207}. Generally these two definitions are slightly different,  but on essentially non-branching spaces they coincide (see for instance Appendix A in \cite{MR3691502} or Proposition 9.1 in \cite{MR4309491}). We adopt the one given in \cite{MR2341840}.
\begin{definition}\label{define2-2}
	Let $K\in\R$ and $N\in [1,\infty)$. A metric measure space $(X,d,\mm)$ satisfies $\mcpkn$ if for any $\mu_0 \in \mathscr{P}_2(X)$ of the form
\begin{equation*}
	\mu_0=\frac{1}{\mm(A)}\mm\llcorner_A
\end{equation*}
    for some Borel set $A\subset X$ with $\mm(A)\in(0,\infty)$, and any $o\in X$ there exists $\nu \in \mathrm{OptGeo}(\mu_0,\delta_o)$ such that:
\begin{equation*}
	\frac{1}{\mm(A)}\mm\geq  (e_t)_\sharp \left(\tau^{(1-t)}_{K,N}(d(\gamma_0,\gamma_1))^N\nu(\mathrm{d}\gamma) \right) \quad \forall t\in [0,1].
\end{equation*}
\end{definition}

	 From \cite{MR4309491}, we know that in the setting of essentially non-branching spaces, Definition \ref{define2-2} is equivalent to the following condition: for all $\mu_0,\mu_1\in \mathscr{P}_2(X)$ with $\mu_0\ll \mm$, there exists a unique $\nu \in \mathrm{OptGeo}(\mu_0,\mu_1)$ such that for all $t \in [0,1)$, $\mu_t=(e_t)_\sharp \nu\ll \mm$ and 
\begin{equation}\label{mcp}
	\rho_t^{-\frac{1}{N}}(\gamma_t)\geq	\tau_{K,N}^{(1-t)}(d(\gamma_0,\gamma_1))\rho_0^ {-\frac{1}{N}}(\gamma_0),\quad \text{for}\,\,\nu\text{-}a.e.\,\gamma \in \text{Geo}(X),
\end{equation}
    where $\mu_t=\rho_t\mm$.

\begin{definition}
	Let $K\in\R$ and $N\in [1,\infty)$.  We say that a metric measure space  is a volume cone if there exists $x_0\in X$,  such that
\begin{equation*}
	\frac{\mm(B_{R}(x_0))}{\mm(B_{r}(x_0))}=\left(\frac{R}{r}\right)^N, \quad\forall 0<r<R, 
\end{equation*}
    where $B_{r}(x_0):=\{x\in X:d(x_0,x)<r\}$.
\end{definition}
     
    In order to prove our main results, we need the following  theorem about needle decomposition (here we only consider the case $K=0$).
\begin{theorem}\label{theorem3}
	Let $(X,d,\mm)$ be an essentially non-branching ${\rm MCP}(0,N)$ metric measure space  for some $N\in(1,\infty)$. Then for any $1$-Lipschitz function u: X $\rightarrow \R$, there exists an $\mm$-measurable transport subset $\mathcal{T} \in X$ and a family \{$X_\alpha\}_{\alpha\in Q}$ of subsets of $X$, such that there exists a disintegration of $\mm\llcorner_{\mathcal{T}}$ on \{$X_\alpha\}_{\alpha\in Q}$:
\begin{equation*}
	\mm\llcorner_{\mathcal{T}}=\int_Q \mm_\alpha\, \mathfrak{q} (\mathrm{d} \alpha),\qquad \mathfrak{q}(Q)=1,
\end{equation*}
	and for $\mathfrak{q}$-a.e. $\alpha\in Q$: 
\begin{enumerate}
	\item $X_\alpha$ is a closed geodesic in $(X,d)$,
	\item $\mm_\alpha$ is a Radon measure supported on $X_\alpha$ with $\mm_\alpha=h_\alpha \mathcal{H}^1\llcorner_{X_\alpha}\ll \mathcal{H}^1\llcorner_{X_\alpha}$,
	\item The metric measure space $(X_\alpha,d,\mm_\alpha)$ verifies ${\rm MCP}(0,N)$.
\end{enumerate}
	Here $\mathcal{H}^1$ denotes the one-dimensional Hausdorff measure, \{$X_\alpha\}_{\alpha\in Q}$ are called transport rays,   two distinct transport rays can only meet at their extremal points.
\end{theorem}
    It is worth recalling that, if $ h_\alpha$ is an ${\rm MCP}(0,N)$ density on $I\subset \R$, then for all $x_0, x_1 \in I$ and $t\in [0,1] $,
\begin{equation}\label{1dmcp}
	h_\alpha(tx_1+(1-t)x_0)\geq (1-t)^{N-1}h(x_0).
\end{equation}
    
\begin{remark}\label{remark3}
	If we take $u$ as the distance function $d_{x_0}$, then  for $\mathfrak{q}$-a.e.$\alpha \in Q$, $X_\alpha$ has an extremal point $x_0$.  For more discussions  we refer to Cavalletti--Mondino \cite{MR4175820}.
\end{remark}

 At the end of this part,  we recall the \emph{Generalized Bishop--Gromov volume growth inequality} (cf.  \cite[Remark 5.3]{MR2237207}).
    
\begin{theorem}\label{theorem4}
	Assume that  $(X,d,\mm)$ satisfies ${\rm MCP}(0,N)$ for some $N>1$. Then for any $x\in X$, 
\begin{equation}\label{equation2-2}\tag{GBGI}
	\frac{\mm(B_{r}(x))}{r^N}\geq \frac{\mm(B_{R}(x))}{R^N},\quad \forall 0<r<R.
\end{equation}
\end{theorem}
\section{Heisenberg--Pauli--Weyl uncertainty principle}\label{section3}
    In this section, we will prove  Theorem \ref{theorem1} in five steps. We obtain the Heisenberg--Pauli--Weyl uncertainty principle on volume cone by using needle decomposition, which reduces the problem to one-dimensional metric measure spaces. Conversely, we deduce the rigidity using Generalized Bishop--Gromov inequality.\medskip
\begin{proof}[Proof of Theorem \ref{theorem1}]
	
	~\\
	${\bf (b) \Rightarrow (a)}$:
	
    {\it Step 1}. Assume that the vertex of the volume cone is $O$ and let $d(x):=d_O(x)$. By Theorem  $\ref{theorem3}$ (see also Remark $\ref{remark3}$), we have
\begin{equation}\label{equation3-1}
    \mm\llcorner_{\mathcal{T}}=\int_Q \mm_\alpha\, \mathfrak{q} (\mathrm{d} \alpha),\qquad \mathfrak{q}(Q)=1,\qquad \mm_\alpha=h_\alpha \mathcal{H^1}\llcorner_{X_\alpha},
\end{equation}
    and for $\mathfrak{q}$-a.e.$\alpha \in Q$, $X_\alpha$ has  $O$ as an extremal point.
    
    Consider the optimal transport problem between the probability measures 
\begin{equation*}
	\mu_0=\frac{1}{\mm(B_R(O))}\mm\llcorner_ {B_R(O)},\qquad\text{and}~\qquad \mu_1=\delta_O.
\end{equation*}
    By Definition \ref{define2-2},  there exists a unique geodesic $(\mu_t)_{t\in [0,1]} $ in $(\mathscr{P}_2(X),W_2)$ connecting $\mu_0$ and $\mu_1$,  and there is a measure $\nu \in \mathscr{P}(\text{Geo}(X))$, so that $(e_t)_\sharp \nu=\mu_t$ for all $t\in [0,1)$.  By measure contraction property \eqref{mcp} we can see that $\mu_t$ is concentrated on a subset of $B_{(1-t)R}(O)$, with measure no less than $(1-t)^N \mm(B_R(O))$.  Since  $(X, d, \mm)$ is a volume cone,  $\mu_t$ must has full support in $B_{(1-t)R}(O)$. Since $R>0$ is arbitrary,  we can see that $\mm(X\setminus {\mathcal{T}})=0$, and for $\mathfrak{q}$-a.e.$\alpha \in Q$,   $(X_\alpha,d,\mm_\alpha)$ is also a volume cone.
    
   Combining \eqref{equation2-2} and \eqref{1dmcp}, we know that for $\mathfrak{q}$-a.e.$\alpha\in Q$, $(X_\alpha,d,\mm_\alpha)$ can be identified with  $([0,\infty),|\cdot|,c_{\alpha}x^{N-1}\mathrm{d}x)$, where $c_\alpha$ is a positive constant.\medskip
    
    {\it Step 2}. Fix $\alpha \in Q$, denote by $\tilde{\Delta}_\alpha$ the weighted Laplacian
\begin{equation*}
	\tilde{\Delta}_\alpha:=\Delta -\langle \nabla V_\alpha,\nabla \cdot\rangle
\end{equation*}
    where $V_\alpha(x)$ is given by $e^{-V_\alpha(x)}=c_\alpha x^{N-1}$, and $\Delta, \nabla \cdot$ are understood as directional derivatives in the usual sense.
    
    Fix $u\in$ {Lip}$_c(X, d)\setminus\{0\}$.  On one hand,  for $d(x)=x$ we have $\tilde{\Delta}_\alpha(d^2)=2N$ and
\begin{equation}\label{equation3-2}  
\begin{aligned}
	\left(\int_{X_\alpha} \tilde{\Delta}_\alpha(d^2) u^2 \, \mathrm{d}\mm_\alpha\right)^2
	&=4N^2\left(\int_{X_\alpha} u^2\, \mathrm{d}\mm_\alpha\right)^2.
\end{aligned}
\end{equation} 
    On the other hand, by integration by parts
\begin{equation}\label{equation3-3}  
\begin{aligned}
    \left(\int_{X_\alpha} \tilde{\Delta}_\alpha(d^2) u^2 \, \mathrm{d}\mm_\alpha\right)^2
    &=\left(-\int_0^\infty \left<\nabla (u^2),\nabla (d^2)\right>  \, \mathrm{d}\mm_\alpha \right)^2\\
    &=\left(\int_0^\infty -4ud\left<\nabla u,\nabla d \right> \, \mathrm{d}\mm_\alpha \right)^2\\
 \text{Cauchy-Schwartz}~~ &\leq 16\left(\int_{X_\alpha} d^2 u^2\, \mathrm{d}\mm_\alpha\right) \left(\int_{X_\alpha} |\nabla u|^2\, \mathrm{d}\mm_\alpha\right).
\end{aligned}
\end{equation}
  
    Combining (\ref{equation3-2}) and (\ref{equation3-3}), we obtain
\begin{equation}\label{equation3-4} 
	\left(\int_{X_\alpha} |\nabla u|^2\, \mathrm{d}\mm_\alpha\right) \left(\int_{X_\alpha} d^2 u^2 \, \mathrm{d}\mm_\alpha\right)\geq \frac{N^2}{4}\left(\int_{X_\alpha} u^2\, \mathrm{d}\mm_\alpha\right)^2.\medskip
\end{equation}	

    {\it Step 3}. Denote $C_\alpha:=\int_{X_\alpha} d^2 u^2 \,\mm_\alpha(\mathrm{d}x)>0$,\, $\widetilde{\mm}_\alpha (\mathrm{d}x):=\frac{\mm_\alpha (\mathrm{d}x)}{C_\alpha}$,\, $\widetilde{\mathfrak{q}}(\mathrm{d} \alpha):=C_\alpha \mathfrak{q}(\mathrm{d} \alpha)$.  We can see that $\{\widetilde{\mm}_\alpha\}_{\alpha\in Q}$ is also a disintegration of $\mm$
\begin{equation*}
	\int_Q \widetilde{\mm}_\alpha\, \widetilde{\mathfrak{q}}(\mathrm{d} \alpha)=\int_Q \frac{\mm_\alpha}{C_\alpha}\cdot C_\alpha \mathfrak{q} (\mathrm{d} \alpha)=\int_Q \mm_\alpha\, \mathfrak{q} (\mathrm{d} \alpha)=\mm.  
\end{equation*}

    Fix $\alpha \in Q$, multiplying $1/{C^2_\alpha}$ on both sides of $(\ref{equation3-4})$, we have
\begin{equation}\label{equation3-5}
	\left(\int_{X_\alpha} |\nabla u|^2\,\d\widetilde{\mm}_\alpha \right) \left(\int_{X_\alpha} d^2 u^2 \,\d\widetilde{\mm}_\alpha \right)\geq \frac{N^2}{4}\left(\int_{X_\alpha} u^2 \,\d\widetilde{\mm}_\alpha\right)^2.
\end{equation}	
Note that
\begin{equation}\label{equation3-6}
\begin{aligned}
	&\int_{X_\alpha} d^2 u^2 \,\d\widetilde{\mm}_\alpha\equiv1, \quad  \forall \alpha \in Q,\\
	\int_X d^2 u^2& \mathrm{d}\mm=\int_Q\int_{X_\alpha} d^2 u^2 \,\widetilde{\mm}_\alpha(\mathrm{d}x) \widetilde{\mathfrak{q}}(\mathrm{d} \alpha)=\widetilde{\mathfrak{q}}(Q).
\end{aligned}
\end{equation}
    Combining above identities and integrating $\alpha$ on both sides of (\ref{equation3-5}), we conclude that 
\begin{equation}\label{equation3-7}
	\int_Q \left(\int_{X_\alpha} |\nabla u|^2\,\widetilde{\mm}_\alpha(\mathrm{d}x) \right)\widetilde{\mathfrak{q}}(\mathrm{d} \alpha)\geq \frac{N^2}{4}\int_Q \left(\int_{X_\alpha} u^2 \,\widetilde{\mm}_\alpha(\mathrm{d}x)\right)^2\widetilde{\mathfrak{q}}(\mathrm{d} \alpha).
\end{equation} 
    Multiplying $\widetilde{\mathfrak{q}}(Q)$ on both sides of (\ref{equation3-7}) and combining 
\begin{equation*}
    	|\nabla u(x)|\leq |\text{lip}\,u(x)|\qquad a.e. \,x\in X_\alpha ,~\forall\alpha \in Q,
\end{equation*}
  by Cauchy-Schwartz inequality, we get
\begin{equation}\label{equation3-8}
\begin{aligned}
    &\left(\int_X |\,\text{lip}\,u(x)\,|^2\, \mm(\mathrm{d}x)\right)
     \left(\int_X d^2 u^2(x)\, \mm(\mathrm{d}x)\right)\\
    &\geq\frac{N^2}{4}\left(\int_Q \left(\int_{X_\alpha} u^2(x) \,\widetilde{\mm}_\alpha(\mathrm{d}x)\right)^2 \widetilde{\mathfrak{q}} (\mathrm{d} \alpha)\right)\cdot\left(\int_Q 1^2\, \widetilde{\mathfrak{q}}(\mathrm{d} \alpha)\right)\\
    &\geq \frac{N^2}{4}\left(\int_Q \int_{X_\alpha} u^2(x) \,\widetilde{\mm}_\alpha(\mathrm{d}x) \widetilde{\mathfrak{q}}(\mathrm{d} \alpha)\right)^2\\
    &=\frac{N^2}{4}\left(\int_X u^2(x)\, \mm(\mathrm{d}x)\right)^2.
\end{aligned}
\end{equation}
    which is \ref{first} with $x_0=O$.
    
    \medskip

    {\it Step 4}. Following Krist\'aly and Ohta \cite{MR3862150,MR3096522} we can prove the sharpness of $\frac{N^2}{4}$. Since $(X,d,\mm)$ is a volume cone with vertex $O$,  we have 
\begin{equation*}
    \frac{\mm(B_R(O))}{\mm(B_r(O))}=\left(\frac{R}{r}\right)^N,\quad \forall 0<r<R.
\end{equation*}
    Without loss of generality, we can assume that
\begin{equation}\label{equation3-9}
    \mm(B_{\rho}(O))=k \omega_N \rho^N, \quad  \forall \rho>0,
\end{equation}
    where $k$ is a positive constant and $\omega_N:=\pi^{\frac{N}{2}}/\Gamma(\frac{N}{2}+1)$  is the volume of the unit ball in $\R^N$ (cf. \cite{MR0717827}).
    
For each $\lambda>0$, consider the sequence of functions $ u_{\lambda,k}: X\rightarrow \R,k\in \N$ defined as
\begin{equation*}
	u_{\lambda,k}(x):=\max\big\{0,\min\{0,k-d(x)\}+1\big\}e^{-\lambda d^2(x)}.
\end{equation*}
  Notice that supp $	u_{\lambda,k}$=\{$x\in X: d(x)\leq k+1$\} and MCP spaces  are locally compact. For any $\lambda>0$ and $k\in \N$, we have $u_{\lambda,k}\in$ Lip$_c(X, d)$ so  it satisfies  \eqref{equation3-8}. Set
\begin{equation*}
	u_\lambda(x):=\lim_{k\rightarrow \infty} u_{\lambda,k}(x)=e^{-\lambda d^2(x)}.
\end{equation*}
A simple approximation procedure based on (\ref{equation3-9}) shows that $u_\lambda$ verifies (\ref{equation3-8}) as well. Next we will prove that $u_\lambda$ attains the equalities in (\ref{equation3-8}).
    
    Firstly consider a function $T:(0,\infty)\rightarrow \R$ defined as 
\begin{equation*}
	T(\lambda)=\int_X e^{-2\lambda d^2} \,\mathrm{d} \mm.
\end{equation*}
    It is well-defined and differentiable, and we have 
\begin{equation*}
	T'(\lambda)=\int_X (-2d^2)e^{-2\lambda d^2} \,\mathrm{d} \mm.
\end{equation*} 
Note that 
\[
|\text{lip}\,u_\lambda|^2=|2\lambda  d\, \text{lip}(d)\,e^{-\lambda d^2}|^2\leq4\lambda^2d^2 e^{-2\lambda d^2}.
\]It is sufficient to check   the following equation:
\begin{equation}\label{equation3-10}
    -\lambda T'(\lambda)=\frac{N}{2}T(\lambda), \qquad \forall \lambda>0,
\end{equation}	
     equivalently,
\begin{equation}\label{equation3-11}
	T(\lambda)=C \lambda^{-\frac{N}{2}}~~\,\text{ for some }C>0.
\end{equation}
In fact, by the layer cake representation (or Cavalieri's formula)  and changing a variable, we have
\begin{align*}
	T(\lambda)&=\int_{0}^{\infty} \mm(\{x\in X:e^{-2\lambda d^2}>t\}) \,\mathrm{d}t \\
	&=4\lambda \int_{0}^{\infty} \mm\big(B_{\rho}(O)\big) \rho e^{-2\lambda \rho^2} \, \mathrm{d} \rho\\
	&=4\lambda k \omega_N \int_{0}^{\infty} \rho^{N+1} e^{-2\lambda \rho^2}\, \mathrm{d} \rho\\ \text{set}~2\lambda \rho^2=y~~~~
	&=\frac{1}{(2\lambda)^{\frac{N}{2}}} k \omega_N \int_{0}^{\infty} y^\frac{N}{2} e^{-y}\, \mathrm{d}y\\
	&=\frac{1}{(2\lambda)^{\frac{N}{2}}} k \omega_N \Gamma(\frac{N}{2}+1)\\
	&=(\frac{\pi}{2})^{\frac{N}{2}} k \lambda^{-\frac{N}{2}}
\end{align*}
 which is the thesis.
     
     \medskip
    ${\bf (a) \Rightarrow (b)}$:
    
    {\it Step 5}. By Theorem \ref{theorem4} we know $\rho \mapsto \frac{\mm(B_{\rho}(x_0))}{\rho^N}$ is non-increasing on $(0, \infty)$.  Without loss of generality, we assume
\begin{equation}\label{equation3-12}
    \lim_{\rho \rightarrow 0^+} \frac{\mm(B_{\rho}(x_0))}{\omega_N \rho^N}=A,\quad A>0\,\,\text{is a finite constant},
\end{equation}
so that
\begin{equation}\label{equation3-13}
	\mm(B_{\rho}(x_0))\leq A \omega_N \rho^N,\qquad \forall \rho >0.
\end{equation}

Consider the function
\begin{equation}\label{equation3-14}
    T(\lambda)=4\lambda k \omega_N \int_{0}^{\infty} \rho^{N+1} e^{-2\lambda \rho^2}\, \mathrm{d} \rho=\frac{2}{(2\lambda)^{\frac{N}{2}}} k \omega_N \int_{0}^{\infty} t^{N+1} e^{-t^2}\, \mathrm{d}t,
\end{equation} 
this function certainly satisfies (\ref{equation3-10}).
    
    Fix $x_0 \in X $, we consider the class of functions
\begin{equation}\label{equation3-15}
	\widetilde{u}_\lambda(x)=e^{-\lambda d_{x_0}^2(x)},\quad \lambda>0.
\end{equation}
Similarly, we can approximate $\widetilde{u}_\lambda$ by elements in Lip$_c(X, d)$. Inserting $\widetilde{u}_\lambda$ into \ref{first}, we obtain 
\begin{equation}\label{equation3-16}
    2\lambda \int_X d^2_{x_0} e^{-2\lambda d_{x_0}^2} \,\mathrm{d} \mm\geq\frac{N}{2} \int_X e^{-2\lambda d_{x_0}^2} \,\mathrm{d} \mm,\quad \lambda>0.
\end{equation}
    We introduce a function $P: (0,\infty)\rightarrow \R$ defined by
\begin{equation*}
	P(\lambda)=\int_X e^{-2\lambda d_{x_0}^2} \,\mathrm{d} \mm,\quad \lambda>0.
\end{equation*}
    It is well-defined and differentiable.
    By the layer cake representation, the function can be equivalently rewritten to 
\begin{equation*}
	P(\lambda)=\int_{0}^{\infty} \mm\big(\{x\in X:e^{-2\lambda d_{x_0}^2}>t\}\big) \,\mathrm{d}t=4\lambda \int_{0}^{\infty} \mm(B_{\rho}(x_0))\cdot \rho e^{-2\lambda \rho^2} \, \mathrm{d} \rho.
\end{equation*}
     Thus relation (\ref{equation3-16}) is equivalent to 
\begin{equation}\label{equation3-17}
	-\lambda P'(\lambda) \geq \frac{N}{2}P(\lambda), \quad  \lambda>0.
\end{equation}
    We claim that 
\begin{equation*}
	P(\lambda)\geq \frac{A}{k}T(\lambda), \qquad \forall \lambda>0.
\end{equation*}
    By (\ref{equation3-10}) and (\ref{equation3-17}), it turns out that 
\begin{equation*}
	\frac{P'(\lambda)}{P(\lambda)} \leq \frac{T'(\lambda)}{T(\lambda)},\quad \lambda>0. 
\end{equation*}
    Integrating this inequality, it yields that the function  $\lambda\mapsto \frac{P(\lambda)}{T(\lambda)} $ is non-increasing; in particular, for every $\lambda>0$,
\begin{equation*}
	\frac{P(\lambda)}{T(\lambda)}\geq\liminf_{\lambda\rightarrow \infty }\frac{P(\lambda)}{T(\lambda)}.
\end{equation*}
    To prove the claim we only need to show 
\begin{equation}\label{equation3-18}
	\liminf_{\lambda\rightarrow \infty }\frac{P(\lambda)}{T(\lambda)}\geq \frac{A}{k}.
\end{equation}
    Due to (\ref{equation3-12}), for any $\varepsilon>0$ small enough, we can find $\rho_\varepsilon>0$ such that 
\begin{equation*}
	\mm(B_{\rho}(x_0))\geq (A-\varepsilon) \omega_N \rho^N,\quad \forall  \rho \in [0,\rho_\varepsilon].
\end{equation*}
    Consequently, we have
\begin{align*}
	P(\lambda)&=4\lambda \int_{0}^{\infty} \mm(B_{\rho}(x_0))\cdot \rho e^{-2\lambda \rho^2} \, \mathrm{d} \rho\\
	&\geq 4\lambda (A-\varepsilon) \omega_N \int_{0}^{\rho_\varepsilon} \rho^{N+1} e^{-2\lambda \rho^2}\, \mathrm{d} \rho \\ \text{set}~2\lambda \rho^2=t^2~~~~
	&=\frac{2}{(2\lambda)^{\frac{N}{2}}} (A-\varepsilon) \omega_N \int_{0}^{\sqrt{2\lambda}\rho_\varepsilon} t^{N+1} e^{-t^2}\, \mathrm{d}t.
\end{align*}
    Combining (\ref{equation3-14}) we get
\begin{equation*}
	\liminf_{\lambda\rightarrow \infty }\frac{P(\lambda)}{T(\lambda)}\geq \frac{A-\varepsilon}{k}.
\end{equation*}
    Since $\varepsilon>0$ is arbitrary, relation (\ref{equation3-18}) holds and we complete the proof of the claim.

    From the claim we know
\begin{align*}
	4\lambda \int_{0}^{\infty} \left[\mm(B_{\rho}(x_0))-\frac{A}{k} \cdot k \omega_N \rho^N \right] \rho e^{-2\lambda \rho^2} \mathrm{d} \rho \geq 0, \qquad \forall \lambda>0.
\end{align*}
    Due to relation (\ref{equation3-13}), we have
\begin{equation*}
	\mm(B_{\rho}(x_0))=A \omega_N \rho^N, \qquad \forall \rho>0,
\end{equation*} 
 so $(X,d,\mm)$ is a volume cone.
\end{proof}
\section{Caffarelli--Kohn--Nirenberg inequality}\label{section4}
    In this section, we will prove the rigidity of \ref{second}. Its proof  is similar to Theorem \ref{theorem1}, so we   adopt   the same notations as in the proof of Theorem \ref{theorem1} and omit some details.
    
    \medskip
\begin{proof}[Proof of Theorem \ref{theorem2}]	
	
	~	
	
	${\bf (b)\Rightarrow (a)}$:
	
	 {\it Step 1}. 
		Similarly,  for $\mathfrak{q}$-a.e.$\alpha\in Q$,  $(X_\alpha,d,\mm_\alpha)$ can be identified with $([0,\infty),|\cdot|,c_\alpha x^{N-1}\mathrm{d}x)$ and  $ d\tilde{\Delta}_\alpha d=N-1$ on  $(X_\alpha,d,\mm_\alpha)$. 
	
	Fix $u\in {\rm Lip}_c(X, d)$, we have
\begin{equation}\label{equation4-1}
\begin{aligned}
	\int_{X_\alpha} \frac{|u|^p}{d^q} \,\mm_\alpha(\mathrm{d}x)=&\frac{1}{N-1}\int_{X_\alpha} \frac{|u|^p}{d^{q-1}} \tilde{\Delta}_\alpha d\,\mm_\alpha(\mathrm{d}x)\\
	=&-\frac{p}{N-1}\int_{X_\alpha} \frac{|u|^{p-1}}{d^{q-1}} \langle\nabla |u|,\nabla d\rangle\,\mm_\alpha(\mathrm{d}x)\\
	&+\frac{q-1}{N-1} \int_{X_\alpha} \frac{|u|^p}{d^q} |\nabla d|^2 \,\mm_\alpha(\mathrm{d}x).
\end{aligned}   
\end{equation} 
	Note $|\nabla d|=1 \,\,a.e.\,$on $X_\alpha$, a reorganization of the above estimate implies that 
\begin{equation}\label{equ1}
\begin{aligned}
	\frac{N-q}{p}\int_{X_\alpha} \frac{|u|^p}{d^q} \,\mm_\alpha(\mathrm{d}x)
	&\leq \int_{X_\alpha} \frac{|u|^{p-1}}{d^{q-1}} |\nabla u|\,\mm_\alpha(\mathrm{d}x)\\\text{Cauchy--Schwartz}~~
	&\leq \left(\int_{X_\alpha} |\nabla u\,|^2\,\mm_\alpha(\mathrm{d}x)\right)^{\frac{1}{2}}\left(\int_{X_\alpha} \frac{|u|^{2p-2}}{d^{2q-2}}\,\mm_\alpha(\mathrm{d}x)\right)^{\frac{1}{2}}.
\end{aligned}	
\end{equation}

	Similar to the proof of Theorem \ref{theorem1}, we can adjust the decomposition  and obtain 
\begin{equation}\label{equation4-3}
	\left(\int_X |\,\text{lip}\,u\,|^2\,\mathrm{d} \mm\right)\left(\int_X \frac{|u|^{2p-2}}{d^{2q-2}}\,\mathrm{d} \mm\right) \geq \frac{(N-q)^2}{p^2}\left(\int_X \frac{|u|^p}{d^q}\,\mathrm{d} \mm\right)^2
\end{equation}   
	which is the thesis.\medskip 
		
	{\it Step 2}. 
	Similarly, to prove the sharpness of $\frac{(N-q)^2}{p^2}$, we consider  functions 
\begin{equation*}
	u_{\lambda,k}(x):=\max \Big\{0,\min\{0,k-d(x)\}+1 \Big\}\left(\lambda+\max\big\{d(x),\frac{1}{k}\big\}^{2-q}\right)^{\frac{1}{2-p}}.
\end{equation*}
Define
\begin{equation*}
	u_\lambda(x):=\lim_{k\rightarrow \infty} u_{\lambda,k}(x)=(\lambda+d^{2-q})^{\frac{1}{2-p}}.
\end{equation*}	
	By an approximation argument we can prove that $u_\lambda$ verifies (\ref{equation4-3}) as well. Next we will show that $u_\lambda$ attains the equality in (\ref{equation4-3}).
	
	Consider a function $T:(0,\infty)\rightarrow \R$ defined by 
\begin{equation*}
	T(\lambda)=\frac{p-2}{p}\int_X \frac{(\lambda+d^{2-q})^{\frac{p}{2-p}}}{d^q} \,\mathrm{d} \mm.
\end{equation*}
	It is well-defined and differentiable, and we have 
\begin{equation*}
\begin{aligned}
	T'(\lambda)=&-\int_X \frac{(\lambda+d^{2-q})^{\frac{2p-2}{2-p}}}{d^q} \,\mathrm{d} \mm.
\end{aligned}
\end{equation*} 
Notice that
\[
|\text{lip}\,u_\lambda|^2\leq \left(\frac{2-q}{2-p}\right)^2(\lambda+d^{2-q})^{\frac{2p-2}{2-p}} d^{2-2q}.
\]	It is sufficient to prove 
\begin{equation}\label{equation4-4}
	\lambda T'(\lambda)=\left(\frac{N-q}{2-q}-\frac{p}{p-2}\right)T(\lambda), \qquad \forall \lambda>0.
\end{equation}
	Note that $\alpha:=\frac{N-q}{2-q}-\frac{p}{p-2}<0$. It is equivalent to prove
\begin{equation}\label{equation4-5}
	T(\lambda)=C \lambda^\alpha, \quad \forall \lambda>0 \,\,\text{and for some }C>0.
\end{equation}
By the layer cake representation and changing a variable, we have
\begin{align*}
	T(\lambda)&=\frac{p-2}{p}\int_X \frac{(\lambda+d^{2-q})^{\frac{p}{2-p}}}{d^q} \,\mathrm{d} \mm \\
	&=\frac{p-2}{p}\int_{0}^{\infty}\mm(\{x\in X: \frac{(\lambda+d^{2-q})^{\frac{p}{2-p}}}{d^q}>t\})  \,\mathrm{d} t \\
	&=\frac{p-2}{p} k \omega_N \int_{0}^{\infty} \frac{(\lambda+\rho^{2-q})^{\frac{2p-2}{2-p}}}{\rho^{2q-1-N}}\left(\frac{(2-q)p}{p-2}+q(\lambda\rho^{q-2}+1)\right) \, \mathrm{d} \rho\\\text{set}~\rho=\lambda^\frac{1}{2-q}y~~~
	&=\frac{p-2}{p} k \omega_N  \int_{0}^{\infty} \frac{(\lambda+\lambda y^{2-q})^{\frac{2p-2}{2-p}}}{\lambda^{\frac{2q-1-N}{2-q} y^{2q-1-N}}}\left(\frac{(2-q)p}{p-2}+q(y^{q-2}+1)\right)\lambda^\frac{1}{2-q} \, \mathrm{d} y\\
	&=\lambda^\alpha \cdot\frac{p-2}{p} k \omega_N  \int_{0}^{\infty} \frac{(1+ y^{2-q})^{\frac{2p-2}{2-p}}}{ y^{2q-1-N}}\left(\frac{(2-q)p}{p-2}+q(y^{q-2}+1)\right) \, \mathrm{d} y
\end{align*}
	which is (\ref{equation4-5}) and we complete the proof of $(b)\Rightarrow (a)$.
	
	\medskip
	
	${\bf (a)\Rightarrow (b)}$.
	
	{\it Step 3}. Recall that
\begin{equation}\label{equation4-6}
	\lim_{\rho \rightarrow 0^+} \frac{\mm(B_{\rho}(x_0))}{\omega_N \rho^N}=A,\quad A>0\,\,\text{is a finite constant}.
\end{equation} 
	and
\begin{equation}\label{equation4-7}
	\mm(B_{\rho}(x_0))\leq A \omega_N \rho^N,\qquad \forall \rho >0.
\end{equation}
	Similarly, we consider the function
\begin{equation}
\begin{aligned}\label{equation4-8}
	T(\lambda)&=\frac{p-2}{p} k \omega_N \int_{0}^{\infty} \frac{(\lambda+\rho^{2-q})^{\frac{2p-2}{2-p}}}{\rho^{2q-1-N}}\left(\frac{(2-q)p}{p-2}+q(\lambda\rho^{q-2}+1)\right) \, \mathrm{d} \rho\\
	&=\lambda^\alpha \cdot\frac{p-2}{p} k \omega_N  \int_{0}^{\infty} \frac{(1+ y^{2-q})^{\frac{2p-2}{2-p}}}{ y^{2q-1-N}}\left(\frac{(2-q)p}{p-2}+q(y^{q-2}+1)\right) \, \mathrm{d} y,
\end{aligned}
\end{equation}
		and the class of functions
\begin{equation}\label{equation4-9}
	\widetilde{u}_\lambda(x)=(\lambda+d^{2-q}_{x_0})^{\frac{1}{2-p}},\quad \lambda>0.
\end{equation}
	Similarly, the functions $\widetilde{u}_\lambda$ can be approximated by elements in Lip$_c(X,d)$ for every $\lambda>0$. By inserting $\widetilde{u}_\lambda$ into \ref{second}, we obtain that 
\begin{equation}\label{equation4-10}
	\frac{2-q}{p-2} \int_X \frac{(\lambda+d^{2-q}_{x_0})^{\frac{2p-2}{2-p}}}{d^{2q-2}_{x_0}} \,\mathrm{d} \mm\geq\frac{N-q}{p}\int_X \frac{(\lambda+d^{2-q}_{x_0})^{\frac{p}{2-p}}}{d^q_{x_0}} \,\mathrm{d} \mm.
\end{equation}
	Define a function $P: (0,\infty)\rightarrow \R$  by
\begin{equation*}
	P(\lambda)=\frac{p-2}{p}\int_X \frac{(\lambda+d^{2-q}_{x_0})^{\frac{p}{2-p}}}{d^q_{x_0}} \,\mathrm{d} \mm,\quad \lambda>0.
\end{equation*}
	It is well-defined and differentiable. Through similar arguments, (\ref{equation4-10}) is equivalent to
\begin{equation}\label{equation4-11}
	\lambda P'(\lambda)\geq\left(\frac{N-q}{2-q}-\frac{p}{p-2}\right)P(\lambda), \qquad \forall \lambda>0.
\end{equation}
	Now we claim that 
\begin{equation*}
	P(\lambda)\geq \frac{A}{k}T(\lambda), \qquad \forall \lambda>0.
\end{equation*}
	Due to (\ref{equation4-6}), for any $\varepsilon>0 $ small enough, we can find $\rho_\varepsilon>0$ such that 
\begin{equation*}
	\mm(B_{\rho}(x_0))\geq (A-\varepsilon) \omega_N \rho^N,\quad \forall  \rho \in [0,\rho_\varepsilon].
\end{equation*}
		Therefore, by layer cake representation and changing the variable $\rho=\lambda^\frac{1}{2-q}y$, it turns out that
\begin{align*}
	P(\lambda)&=\frac{p-2}{p}\int_{0}^{\infty}\mm(\{x\in X: \frac{(\lambda+d^{2-q}_{x_0})^{\frac{p}{2-p}}}{d^q_{x_0}}>t\})  \,\mathrm{d} t \\
	&=\frac{p-2}{p} \int_{0}^{\infty} \mm(B_{\rho}(x_0)) \frac{(\lambda+\rho^{2-q})^{\frac{2p-2}{2-p}}}{\rho^{2q-1}}\left(\frac{(2-q)p}{p-2}+q(\lambda\rho^{q-2}+1)\right) \, \mathrm{d} \rho\\
	&\geq \frac{p-2}{p}(A-\varepsilon) \omega_N \int_{0}^{\rho_\varepsilon}  \frac{(\lambda+\rho^{2-q})^{\frac{2p-2}{2-p}}}{\rho^{2q-1-N}}\left(\frac{(2-q)p}{p-2}+q(\lambda\rho^{q-2}+1)\right) \, \mathrm{d} \rho\\
	&=\lambda^\alpha \cdot\frac{p-2}{p} (A-\varepsilon) \omega_N  \int_{0}^{\lambda^\frac{1}{q-2}\rho_\varepsilon} \frac{(1+ y^{2-q})^{\frac{2p-2}{2-p}}}{ y^{2q-1-N}}\left(\frac{(2-q)p}{p-2}+q(y^{q-2}+1)\right) \, \mathrm{d} y.
\end{align*}
	Combining (\ref{equation4-8}) and the fact that $\frac{1}{q-2}<0$ , we have   
\begin{equation*}
	\liminf_{\lambda\rightarrow 0^+ }\frac{P(\lambda)}{T(\lambda)}\geq\frac{A-\varepsilon}{k}.
\end{equation*} 
	Since $\varepsilon>0$ is arbitrary, we  obtain
\begin{equation}\label{equation4-12}
	\liminf_{\lambda\rightarrow 0^+ }\frac{P(\lambda)}{T(\lambda)}\geq\frac{A}{k}.
\end{equation} 
	Combining (\ref{equation4-4}) and (\ref{equation4-11}), we have
\begin{equation}\label{equation4-13}
	\frac{P'(\lambda)}{P(\lambda)} \geq \frac{T'(\lambda)}{T(\lambda)},\quad \forall\lambda>0.
\end{equation}
    Integrating this inequality, it yields that the function  $\lambda\mapsto \frac{P(\lambda)}{T(\lambda)} $ is non-decreasing.  Combining with (\ref{equation4-12}) we get
\begin{equation*}
    \frac{P(\lambda)}{T(\lambda)}\geq \frac{A}{k},\quad \forall \lambda>0,
\end{equation*}
This means, for any $\lambda>0$,
\begin{align*}
   	\frac{p-2}{p} \int_{0}^{\infty} \left(\mm(B_{\rho}(x_0))-\frac{A}{k}\cdot k\omega_N \rho^N\right) \frac{(\lambda+\rho^{2-q})^{\frac{2p-2}{2-p}}}{\rho^{2q-1}}\left(\frac{(2-q)p}{p-2}+q(\lambda\rho^{q-2}+1)\right) \, \mathrm{d}\rho\geq0.
\end{align*}
     Combining with  (\ref{equation4-7}), we have
\begin{equation*}
	\mm(B_{\rho}(x_0))=A \omega_N \rho^N, \qquad \forall \rho>0,
\end{equation*} 
    so $(X,d,\mm)$ is a volume cone and  we complete the proof of Theorem \ref{theorem2}.
\end{proof}
\begin{remark}
	By the same methods, we can prove the rigidity of  more general CKN inequalities in \cite{MR2350131}.  To achieve this, we just need to replace Cauchy--Schwartz with H\"{o}lder inequality in (\ref{equ1}).
\end{remark}


\end{document}